\documentclass[a4paper,10pt]{amsart}

\usepackage[english]{babel}
\usepackage{amsmath,amssymb,amsthm}

\newcommand{\metric}[2]{\ensuremath{\langle #1, #2\rangle}}  
\newcommand{\nks}{\ensuremath{S^3\times S^3}}   
\newcommand{\lcc}{\ensuremath{\tilde \nabla}}   
\renewcommand{\epsilon}{\varepsilon}            
\newcommand{\ijk}{\ensuremath{\epsilon_{ijk}}}  
\renewcommand{\aa}{\ensuremath{\alpha\cdot\alpha}}  
\newcommand{\ab}{\ensuremath{\alpha\cdot\beta}}     
\newcommand{\bb}{\ensuremath{\beta\cdot\beta}}      

\DeclareMathOperator{\re}{Re}    
\DeclareMathOperator{\im}{Im}    

\newtheorem{theorem}{Theorem}[section]   
\newtheorem*{theorem*}{Theorem}          
\newtheorem{lemma}[theorem]{Lemma}
\newtheorem{proposition}[theorem]{Proposition}

\theoremstyle{definition}

\newtheorem{corollary}[theorem]{Corollary}
\newtheorem{example}{Example}
\newtheorem{remark}{Remark}

\title[Almost complex surfaces in~$\nks$]{Almost complex surfaces in the nearly~K\"ahler~$\nks$}
\author{J.~Bolton\and F.~Dillen\and B.~Dioos\and L.~Vrancken}
\address{John Bolton, Durham University, Department of Mathematical Sciences,
  Science Laboratories,
  South Rd.,
  Durham DH1 3LE}
\email{John.Bolton@durham.ac.uk}
\address{Franki Dillen, Bart Dioos, Luc Vrancken,
  KU\ Leuven, Departement Wiskunde, Celestijnenlaan 200B,
  3001 Leuven, Belgium}
\email{Franki.Dillen@wis.kuleuven.be, Bart.Dioos@wis.kuleuven.be}
\address{Luc Vrancken, LAMAV, Universit\'e de Valenciennes, Campus du Mont Houy, 59313 Valenciennes Cedex 9, France}
\email{Luc.Vrancken@univ-valenciennes.fr}

\begin{document}

\begin{abstract}

 In this paper we initiate the study of almost complex surfaces of the nearly K\"ahler~$\nks$. 
We show that on such a surface it is possible to define a global holomorphic differential, 
which is induced by an almost product structure on the nearly K\"ahler~$\nks$. 
We also find a local correspondence between almost complex surfaces in the  nearly K\"ahler~$\nks$ and
 solutions of the general $H$-system equation introduced by Wente (\cite{wente}), thus obtaining
 a geometric interpretation of solutions of the general $H$-system equation. 
From this we deduce a correspondence between constant mean curvature surfaces in $\mathbb R^3$ and
 almost complex surfaces in the nearly K\"ahler~$\nks$ with vanishing holomorphic differential.
 This correspondence allows us to obtain a classification of the totally geodesic almost complex surfaces. 
Moreover, we prove that almost complex topological 2-spheres in~$\nks$ are totally geodesic. 
Finally, we also show that every almost complex surface with parallel second fundamental form is totally geodesic.
\end{abstract}

\maketitle

\section*{Introduction}

Nearly K\"ahler manifolds have been studied intensively in the 1970's by Gray~\cite{graynearlykahler}.
These nearly K\"ahler manifolds are almost Hermitian manifolds with almost complex structure~$J$ for which the tensor field~$\nabla J$
is skew-symmetric. In particular, the almost complex structure is non-integrable if the manifold
is non-K\"ahler. A well known example is the nearly K\"ahler 6-dimensional sphere, whose almost complex structure $J$
can be defined in terms of the vector cross product on~$\mathbb{R}^7$. Recently it has been shown
by Butruille~\cite{butruille} that the only homogeneous 6-dimensional nearly K\"ahler manifolds are
the nearly K\"ahler 6-sphere, $\nks$, the projective space~$\mathbb{C}P^3$ and the flag manifold $SU(3)/U(1)\times U(1)$.
All these spaces are compact 3-symmetric spaces.

There are two natural types of submanifolds of nearly K\"ahler (or more generally, almost Hermitian) manifolds,
namely almost complex and totally real submanifolds. Almost complex submanifolds are submanifolds
whose tangent spaces are  invariant under $J$. Almost complex submanifolds
in the nearly K\"ahler manifold $S^6$ have been studied by many authors
(see~e.g.\ \cite{boltonpeditwoodward}, \cite{bolton}, \cite{bryant}, \cite{dillenalmostcomplex}, \cite{grayalmostcomplex}, \cite{sekigawa}). 
Also in the nearly K\"ahler $\mathbb CP^3$ some results have been obtained in \cite{xufeng}.

In this paper we initiate the study of almost complex submanifolds of~$\nks$.
Six-dimensional non-K\"ahler nearly K\"ahler manifolds do not admit 4-dimensional almost complex
submanifolds (\cite{podesta}), so the almost complex submanifolds are surfaces.
  The paper is organized as follows: the basics on nearly K\"ahler manifolds and submanifold theory will
be recapitulated in the first section.
In Section~2 we will discuss the nearly K\"ahler structure and an almost product structure~$P$ on~$\nks$. Whereas in the previous works of a.o. \cite{butruille} the structure is presented in terms of Lie groups, here we will present everything using the classical structure on $S^3 \times S^3$. This allows us to remark that the nearly K\"ahler metric, up to a constant factor, corresponds to the Hermitian metric associated to the standard metric on  $S^3 \times S^3$.
In Section~3 it will be shown that to every simply connected almost complex
surface~$M$ in~$\nks$ one can associate a surface in Euclidean 3-space. This associated surface~$\epsilon$
satisfies the $H$-surface equation
\[
  \epsilon_{uu} + \epsilon_{vv} = -\frac{4}{\sqrt{3}} \epsilon_u \times \epsilon_v,
\]
see \cite{wente}. Note that this correspondence works in both directions. This equation also implies that $g(P\phi_z,\phi_z)\,dz^2$ is a holomorphic differential.
Furthermore, under the assumption~$PTM\subset T^\perp M$, i.e. if the holomorphic differential vanishes, the $H$-surface has constant mean curvature.
These results enable us to prove the following theorems.
\begin{theorem*}
 If $M$ is an almost complex surface of $\nks$  with parallel second fundamental form, then $M$ is totally geodesic.
\end{theorem*}
\begin{theorem*}
 An almost complex topological 2-sphere~$S^2$ in the nearly~K\"ahler~$\nks$ is totally geodesic.
\end{theorem*}
The latter result marks a difference from the case of the nearly K\"ahler 6-sphere: there exists
an immersion from~$S^2(1/6)$ in $S^6$ which is not totally geodesic (see \cite[\S\,5, Example 2]{sekigawa}).

In the final section, we give two examples of totally geodesic almost complex surfaces
in~$\nks$. In the first example $P$ maps tangent vectors to tangent vectors; 
in the second one~$P$ maps tangent vectors into normal ones. Furthermore we show that any almost complex surface with
parallel second fundamental form is locally~congruent to one of these two examples.

\section{Preliminaries}
\label{sec:intro}

An almost Hermitian manifold~$(\tilde M,g,J)$ is a manifold endowed with an almost complex structure~$J$
that is compatible with the metric~$g$, i.e.\ an endomorphism $J\colon T \tilde M \to T\tilde M$
such that $J_p^2=-\mathrm{Id}$ for every $p\in \tilde M$ and  $g(JX,JY)=g(X,Y)$ for all $X,Y \in T \tilde M$.
A nearly K\"ahler manifold is an almost Hermitian manifold with the extra condition that
the $(1,2)$-tensor field $G= \lcc J$ is skew-symmetric:
\[
(\lcc_X J)Y + (\lcc_Y J)X = 0  \qquad \text{for every $X$, $Y \in T\tilde M$}.
\]
Here $\lcc$ stands for the Levi-Civita connection of the metric~$g$.
A number of properties hold for this tensor field (\cite{belgun}, \cite{graynearlykahler}):
\begin{align}
 G(X,Y) + G(Y,X)=0,\\
 G(X,JY) + JG(X,Y) = 0,\\
 g(G(X,Y),Z) + g(G(X,Z),Y)=0, \\
 \overline \nabla J = 0.
\end{align}
The canonical Hermitian connection~$\overline\nabla$ is
defined by $\overline\nabla_X Y = \lcc_X Y + \frac{1}{2}(\lcc_X J)JY$.

An almost complex surface $M$ of a nearly~K\"ahler~manifold~$\tilde M$
 is a 2-dimensional submanifold such that the tangent bundle of $M$
is invariant under the almost complex structure, i.e.~$JT M = TM$. 
We denote the Levi-Civita connection on~$M$ by~$\nabla$ and the normal connection on the normal bundle~$T^\perp M$ by~$\nabla^\perp$. 
The formulas of Gauss and Weingarten then are
\begin{align*}
 \lcc_X Y &= \nabla_X Y + h(X,Y),\\
 \lcc_X \xi &= -A_\xi X + \nabla_X^\perp \xi,
\end{align*} 
for tangent vectors~$X$,~$Y$ and a normal vector~$\xi$.
The second fundamental form~$h$ and the shape operator~$A_\xi$ are related to each other
by 
\[
  g(h(X,Y),\xi) = g(A_\xi X,Y).
\]
The Gauss and Weingarten formulas and the properties of~$G$ imply 
\begin{align}
 \nabla_X JX &= J\nabla_X X,           &      h(X,JY) &= Jh(X,Y), \label{eq:ac1}\\
 A_{J\xi} X &= JA_\xi X = - A_\xi JX,  &      G(X,\xi) &= \nabla^\perp_X J\xi - J\nabla_X^\perp \xi. \label{eq:ac2}
\end{align}
see e.g.~\cite{dillenalmostcomplex} or~\cite{sekigawa}.
As an immediate corollary, $M$ itself is nearly K\"ahler and minimal.
Moreover, since each tangent space $T_p M$ is spanned by a unit vector~$X$ and $JX$, 
$G(X,Y)=0$ for every $X,Y \in TM$ and thus $M$ is K\"ahler.

We denote the curvature tensor of~$\lcc$,~$\nabla$ and~$\nabla^\perp$ by~$\tilde R$,~$R$ and~$R^\perp$
respectively. The equations of Gauss, Codazzi and~Ricci then are
\begin{align*}
 &R(X,Y)Z = \bigl(\tilde R(X,Y)Z\bigr)^\top + A_{h(Y,Z)}X - A_{h(X,Z)}Y,\\
 &(\nabla h)(X,Y,Z) - (\nabla h)(Y,X,Z) = \bigl(\tilde R(X,Y)Z\bigr)^\perp,\\
& g\bigl(R^\perp(X,Y)\xi,\eta\bigr) = g\bigl(\tilde  R(X,Y)\xi,\eta\bigr) + g\bigl([A_\xi,A_\eta]X,Y\bigr),
\end{align*}
where~$X,Y,Z\in TM$, ~$\xi, \eta\in T^\perp M$ and $\nabla h$ is defined by
$\nabla^\perp_X h(Y,Z) -h(\nabla_X Y, Z) - h(Y,\nabla_X Z)$. 
A submanifold is called parallel if~$\nabla h$ is zero everywhere.
The second derivative~$\nabla^2 h$ of~$h$ is defined in a similar way by 
\begin{align*}
(\nabla^2 h)(X,Y,Z,W) &= \nabla^\perp_X (\nabla h)(Y,Z,W) -(\nabla h)(\nabla_X Y,Z,W)\\
                          &\quad -(\nabla h)(Y,\nabla_X Z,W) - (\nabla h)(Y,Z,\nabla_X W).
\end{align*}
The Ricci identity for~$\nabla^2 h$ then says
\begin{align*}
 (\nabla^2 h)(X,Y,Z,W)-(\nabla^2 h)(Y,X,Z,W) &= R^\perp(X,Y)h(Z,W) \\
                                   & \quad - h(R(X,Y)Z,W) - h(Z,R(X,Y)W).
\end{align*}
Note that the left hand side vanishes if~$M$ is parallel.

\section{The nearly K\"ahler structure on $\nks$}

We consider the $3$-sphere in $\mathbb{R}^4$ as the set of all unit quaternions.
The vector fields $X_1$,~$X_2$ and~$X_3$ given by
\begin{align*}
 X_1(p) &= pi  =   - x_2 + x_1 i + x_4 j - x_3 k, \\
 X_2(p) &= pj  =   - x_3 - x_4 i + x_1 j + x_2 k, \\
 X_3(p) &= -pk =     x_4 - x_3 i + x_2 j - x_1 k
\end{align*}
at the point $p=x_1 + x_2 i + x_3 j + x_4 k$ form a basis of tangent vector fields.
Thus a tangent vector in $T_p S^3$ can be expressed as $p\alpha$ where $\alpha$ is an imaginary
quaternion.
Using the quaternion relations~$ij=k$, $jk=i$ and $ki=j$ one shows that the Lie brackets are given by $[X_i,X_j]= -2\ijk X_k$. 
Here $\ijk$ is the Levi-Civita symbol.

Using the natural identification $T_{(p,q)}(\nks) \cong T_p S^3 \oplus T_q S^3$, we will
write a tangent vector at $(p,q)$ as  $Z(p,q) = \bigl( U(p,q), V(p,q)\bigr)$ or simply $Z=(U,V)$.
Define the vector fields
\begin{align*}
 E_1(p,q) &= (pi,0),   &   F_1(p,q) &= (0,qi),\\
 E_2(p,q) &= (pj,0),   &   F_2(p,q) &= (0,qj),\\
 E_3(p,q) &= -(pk,0),  &   F_3(p,q) &= -(0,qk).
\end{align*}
These vector fields are mutually orthogonal with respect to the usual product metric
on~$\nks$. The Lie brackets are $[E_i,E_j]=-2\ijk E_k$, $[F_i,F_j]=-2\ijk F_k$ and~$[E_i,F_j]=0$.

The almost complex structure $J$ on $\nks$ is defined as
\[
   JZ(p,q) = \frac{1}{\sqrt{3}}\left( 2pq^{-1}V - U, -2qp^{-1}U + V \right)
\]
for $Z \in T_{(p,q)}(S^3\times S^3)$ (see \cite{butruille}).
Furthermore, we define another metric~$g$ on $\nks$ by
\begin{align*}
   g(Z,Z') &= \frac{1}{2} \left(\metric{Z}{Z'} + \metric{JZ}{JZ'}\right)\\
          &= \frac{4}{3} \left(\metric{U}{U'} +  \metric{V}{V'}\right)
             -\frac{2}{3} \left(\metric{p^{-1}U}{q^{-1}V'} +  \metric{p^{-1}U'}{q^{-1}V}\right).
\end{align*}
where $Z=(U,V)$, $Z'=(U',V')$ and $\metric{\cdot\,}{\cdot}$ is the product metric on~$\nks$.
By definition the almost complex structure is compatible with the metric~$g$.
An easy calculation gives
$g(E_i, E_j) = 4/3  \,\delta_{ij}$,
$g(E_i, F_j) = -2/3 \,\delta_{ij}$ and
$g(F_i, F_j) = 4/3  \,\delta_{ij}$.
Note that this metric differs up to a constant factor from the one introduced in \cite{butruille}. 
Here we set everything up so that it equals the Hermitian metric associated with the usual metric. 
In \cite{butruille}, the factor was chosen in such a way that the standard basis $E_1,E_2,E_3,F_1,F_2,F_3$ has volume 1.

\begin{lemma}
\label{lem:levicivita}
The Levi-Civita connection $\tilde \nabla$ on $\nks$ with respect to the metric~$g$ is given by
\begin{align*}
 \lcc_{E_i} E_j &= -\ijk E_k,                     &   \lcc_{E_i} F_j &= \frac{\ijk}{3}(E_k - F_k),\\
 \lcc_{F_i} E_j &= \frac{\ijk}{3} (F_k - E_k),    &   \lcc_{F_i} F_j &= -\ijk F_k.
\end{align*}
\end{lemma}
\begin{proof}
  Using the Koszul formula, one finds
  \begin{align*}
   g(\lcc_{E_i} E_j,E_k) &= -\frac{4}{3}\ijk,  & g(\lcc_{F_i} E_j,E_k) &= -\frac{2}{3}\ijk,\\
   g(\lcc_{E_i} E_j,F_k) &= \frac{2}{3}\ijk,   & g(\lcc_{F_i} E_j,F_k) &= \frac{2}{3}\ijk,\\
   g(\lcc_{E_i} F_j,F_k) &= -\frac{2}{3}\ijk, & g(\lcc_{F_i} F_j,F_k) &= -\frac{4}{3}\ijk.
  \end{align*}
   Elementary linear algebra then gives the equations hereabove.
\end{proof}

Now one can verify that
\begin{equation}
\label{eq:G}
\begin{split}
 (\lcc_{E_i} J)E_j &= -\frac{2}{3\sqrt{3}}\ijk (E_k + 2F_k), \\
 (\lcc_{E_i} J)F_j &= -\frac{2}{3\sqrt{3}}\ijk (E_k - F_k), \\
 (\lcc_{F_i} J)E_j &= -\frac{2}{3\sqrt{3}}\ijk (E_k - F_k), \\
 (\lcc_{F_i} J)F_j &=  \phantom{-}\frac{2}{3\sqrt{3}}\ijk (2E_k +F_k).
\end{split}
\end{equation}
Hence $\nks$ is nearly K\"ahler, meaning that the tensor field~$G=\lcc J$
is skew-symmetric.

For unitary quaternions~$a$,~$b$ and~$c$, the map~$F\colon \nks\to\nks$ given by
$(p,q)\mapsto (apc^{-1},bqc^{-1})$ is an isometry of $(\nks,g)$ (cf.\ remark after Lemma~2.2 in \cite{podesta}).
Indeed, $F$ preserves the almost complex structure, since
\begin{align*}
 J dF_{(p,q)}(v,w) &= \frac{1}{\sqrt{3}}\bigl(2(apc^{-1})(cq^{-1}b^{-1})bwc^{-1} - avc^{-1}\bigr.,\\
                         &\quad\quad\quad  \bigl. -2(bqc^{-1})(cp^{-1}a^{-1})avc^{-1}+bwc^{-1}\bigr)\\
                   &= dF_{(p,q)}\bigl(J(v,w)\bigr)
\end{align*}
(see also \cite[Proposition~3.1]{moroianu}) and~$F$ preserves the usual metric~$\metric{\cdot\,}{\cdot}$ as well.

Next, we introduce an almost product structure on $\nks$. For a tangent vector~$Z=(U,V)$ at~$(p,q)$, we define
\[
  PZ = (pq^{-1}V, qp^{-1}U).
\]
It is easily seen that
\begin{enumerate}
\item $P^2 = \mathrm{Id}$,
\item $PJ=-JP$
\item $P$ is compatible with the metric~$g$, i.e. $g(PZ,PZ')=g(Z,Z')$. This also implies that $P$ is symmetric with respect to $g$.
\end{enumerate}
Note that~$PE_i = F_i$ and~$PF_i=E_i$.
From these equations and Lemma~\ref{lem:levicivita} it follows that
\begin{equation}
\label{eq:H}
\begin{split}
  (\lcc_{E_i} P)E_j &= \frac{1}{3}\ijk(E_k +2F_k),\\
  (\lcc_{E_i} P)F_j &= -\frac{1}{3}\ijk(2E_k +F_k),\\
  (\lcc_{E_i} P)F_j &= -\frac{1}{3}\ijk(E_k +2F_k),\\
  (\lcc_{F_i} P)F_j &= \frac{1}{3}\ijk(2E_k +F_k).
\end{split}
 \end{equation}
Thus the endomorphism $P$ is not a product structure, i.e.~the tensor field~$H=\lcc P$ does not vanish identically.
However, the almost product structure~$P$ and tensor field $H$ admit the following properties.

\begin{lemma}
For tangent vectors $X$, $Y$ of $\nks$ the following equations hold:
 \begin{gather}
  PG(X,Y) + G(PX,PY) = 0, \label{imp1}\\
  H(X,JY) = JH(X,Y), \label{eq:hj}\\
  G(X,PY) + PG(X,Y) = -2JH(X,Y),\label{imp2}\\
  H(X,PY) + PH(X,Y)=0,\label{eq:h2}\\
  H(X,Y) + H(PX,Y)=0,\label{eq:h3}\\
  \overline\nabla P=0.
 \end{gather}
\end{lemma}
\begin{proof}
 As all expressions are tensorial, one only has to verify them for the basis vectors $E_i$ and $F_j$.
 The first equation can quickly be verified by~\eqref{eq:G} and the fact that $PE_i = F_i$.
 Similarly one can verify equation~\eqref{eq:hj} using~\eqref{eq:H}.
 Equation~\eqref{imp2} follows from \eqref{eq:hj} since
 \begin{align*}
   G(X,PY) + PG(X,Y) &= - H(X,JY) - J H(X,Y)\\
                     &= -2JH(X,Y).
 \end{align*}
 
 The remaining equations are consequences of~\eqref{imp1} and~\eqref{imp2}.
 For instance for~\eqref{eq:h2} we have
 \begin{align*}
 2 \bigl( H(X,PY)+PH(X,Y)\bigr) &= JG(X,Y)+JPG(X,Y) \\
                               & \quad +PJG(X,PY)+PJPG(X,Y)\\
                                &= JG(X,Y) -JG(X,Y) = 0.
 \end{align*}
 Equation~\eqref{eq:h3} can be proven in a similar way. 
 Finally, we have
 \begin{align*}
    (\overline\nabla_X P)Y &= H(X,Y) - \frac{1}{2}\bigl(G(X,PJY)+ PG(X,JY)\bigr) \\
                           &= H(X,Y) +J H(X,JY) = 0.
 \end{align*}
\end{proof}
Note that in the previous lemma, the most fundamental equations are respectively \eqref{imp1} and \eqref{imp2}. 
The first one relates $P$ and $G$, whereas the second one allows us to express $\lcc P$ as a function of~$J$, $P$ and~$\lcc J$. 
It is also elementary to check that $P$ can be expressed in terms of the usual product stucture $Q:Z=(U,V) \mapsto Q(Z)=(-U,V)$ by
\begin{equation} \label{usualproduct}
QJ(Z)= \frac{1}{\sqrt{3}} (-2PZ +Z).\end{equation}
Note however that the usual product structure is not compatible with the metric $g$ and
does not behave nicely with respect to the almost complex structure $J$.

A straightforward, but rather tedious calculation  now shows that the Riemann curvature tensor~$\tilde R$ on $(\nks,g)$ is given by
\begin{equation*}
 \begin{split}
  \tilde R(U,V)W &= \frac{5}{12}\bigl(g(V,W)U - g(U,W)V\bigr) \\
                 &\quad  +\frac{1}{12}\bigl(g(JV,W)JU - g(JU,W)JV - 2g(JU,V)JW\bigr) \\
                 &\quad + \frac{1}{3}\bigl(g(PV,W)PU - g(PU,W)PV  \bigr.\\
                        &\quad  \phantom{\frac{2}{3\sqrt{3}}}\quad\mbox{ } + \bigl. g(JPV,W)JPU - g(JPU,W)JPV\bigr),
 \end{split}
\end{equation*}
and that the tensors $\lcc G$ and $G$ satisfy
\begin{align}
(\lcc G)(X,Y,Z)&= \tfrac 13 (g(X,Z) JY -g(X,Y) JZ -g(JY,Z)X),\\
g\bigl(G(X,Y),G(Z,W)\bigr) &=\tfrac 13 (g(X,Z) g(Y,W)-g(X,W)g(Y,Z)\\
\nonumber &\qquad +g(JX,Z)g(JW,Y)-g(JX,W)g(JZ,Y)).
\end{align}

\begin{remark} Note that we expressed here the new metric $g$ in terms of the standard metric of $S^3 \times S^3$. This can also be reversed. Indeed given $g$, $J$ and $P$, we can define the usual product structure by \eqref{usualproduct} and we can check that the usual metric is given by
$$g(QZ,QZ')+g(Z,Z')= \frac{8}{3} \left(\metric{U}{U'} +  \metric{V}{V'}\right).$$
Hence up to a constant factor the usual metric is the $Q$-compatible metric associated with $g$. \end{remark}

\section{Almost complex surfaces in $\nks$}

We start with some preparatory results. Let us begin
by showing some identities that are similar
to the equations~\eqref{eq:ac1} and~\eqref{eq:ac2} in the preliminaries.
\begin{lemma}\label{lem:PH}
 Let $M$ be an almost complex surface in $\nks$.
 If $PTM = TM$, the following expressions hold for tangent~$X$, $Y$ and
 normal~$\xi$.
 \begin{align*}
 (\nabla_X P)Y&=0  &  A_{P\xi}X &= PA_\xi X = A_\xi PX, \\
 h(X,PY) &= Ph(X,Y), &  H(X,\xi) &= \nabla_X^\perp P\xi - P \nabla_X^\perp \xi.
 \end{align*}
 In particular, $H(X,Y)=0$ and $H(X,\xi)$ is normal to $M$.

 If $PTM \subset T^\perp M$, then
 the second fundamental form $h$ is normal to $PTM$ and $H(X,Y)$ is a normal vector.
\end{lemma}
\begin{proof} First note that from \eqref{imp2} it follows that
\begin{equation}
H(X,Y)=\tfrac{1}{2} \bigl(J G(X,PY)+J PG(X,Y)\bigr.\label{imp3}\bigr).
\end{equation}

We first assume that $P$ maps tangent vectors to tangent vectors. 
In that case $P$ maps normal vectors into normal vectors as well, as $P$ is symmetric and compatible with the metric.

In Section~\ref{sec:intro} we noted that $G(X,Y)=0$ for all $X$,~$Y\in TM$.
Applying the formula of Gauss to \eqref{imp3} together with this fact, we see that
\begin{align*}
0&=H(X,Y)\\
&=\lcc_X P Y - P \lcc_X Y\\
&=\nabla_X PY +h(X, PY) - P \nabla_X Y -P h(X,Y).
\end{align*}
Taking tangent and normal parts gives the first two equations.
The equation $A_{P\xi}=PA_\xi = A_\xi P$ then follows
 easily from the relation $g(h(X,Y),\xi)=g(A_\xi X, Y)$.

Equation~\eqref{eq:ac2} says that~$G(X,\xi)$ is normal. Therefore, since~$J$ and~$P$
map normal vectors into normal vectors, equation~\eqref{imp3} gives that~$H(X,\xi)$ is normal 
as well.
Using the Gauss and Weingarten formulas then gives $H(X,\xi)=\nabla_X^\perp P\xi - P\nabla_X^\perp \xi$.  
This completes the proof in this case.

 Next we assume that $PTM \subset T^\perp M$.  Applying the Gauss and Weingarten formulas to equation~\eqref{eq:hj}
and taking the inproduct with a vector~$JZ\in TM$ gives
 \begin{align*}
  -g(A_{PJY}X, JZ) - g(PJh(X,Y), JZ) &= - g(JA_{PY}X, JZ) + g(Ph(X,JY), JZ) \\
                                    &=  g(A_{PJY}X, JZ) + g(PJh(X,Y), JZ).
 \end{align*} Hence $g(h(X,Z),PY) + g(h(X,Y),PZ) = 0$.
 Since the second fundamental form is symmetric, we obtain $g(h(X,Y),PZ)=0$.
 In a similar way as in the first case one can show that~$H(X,Y)$ is normal. 
 This completes the proof of the lemma. 
\end{proof}
 
\begin{proposition}
 If $M$ is a totally geodesic almost complex surface in $\nks$, then
 either
 \begin{enumerate}
 \item  $P$ maps the tangent space into the normal space and the Gaussian curvature~$K$ is $2/3$
 \item $P$ preserves the tangent space (and therefore also the normal space) and the Gaussian curvature is $0$.
\end{enumerate}
\end{proposition}
\begin{proof}
 Let $p \in M$ be a point of $M$  and $v$ a unit tangent vector to $M$ at  $p$.
 The Codazzi equation implies that $\tilde R(v,Jv)v$ is a tangent vector, thus it must be a multiple of $Jv$.
 By the Gauss equation, we have
 \[
    R(v,Jv)v = \frac{2}{3} \bigl( -Jv + g(PJv,v) Pv - g(Pv,v)PJv \bigr).
 \]
 Moreover, we can choose $v$ such that $g(v,Pv)$ is maximal for all unit vectors in $p$.
 This implies that $g(Pv,Jv)=g(PJv,v)=0$.  The Gauss equation simplifies to
 \begin{equation}
  \label{eq:somegauss}
    R(v,Jv)v = -\frac{2}{3} \bigl( Jv+ g(Pv,v)PJv \bigr).
 \end{equation}
 Now two cases can occur.
 In the first case, if $g(Pv,v)=0$, the Gaussian curvature is~$2/3$. Using $g(Pv,v)=g(Pv,Jv)=0$, $PJ=-JP$ and the fact that $v$ and~$Jv$
span~$T_p M$ we easily get that $PTM \subset T^\perp M$. In the second case, $g(Pv,v)$ is non-zero. Then
 it follows from the Gauss equation~\eqref{eq:somegauss} that $g(Pv,v) PJv$ is a non-zero multiple of $Jv$. Thus $PJv=\pm Jv$, as $P$
 preserves the metric. We may assume $PJv = -Jv$ by replacing~$v$ by~$Jv$ if necessary. Then, since $JP=-PJ$, we find that $Pv= v$ and
 \[
    R(v,Jv,v,Jv) = \frac{2}{3}\bigl(g(Pv,v)^2 - 1\bigr) = 0.
 \]
 This completes the proof.
\end{proof}

The next theorem is a generalization of the previous proposition. 
The idea of the proof is the same as before, but now we apply the Ricci equation
and Ricci identity as well.
\begin{theorem}
 \label{thm:curvature}
 Suppose $M$ is an almost complex surface in $\nks$.
 If $M$ has parallel second fundamental form, then~$PTM=TM$ or~$PTM \subset T^\perp M$.
Moreover,
\begin{enumerate}
\item If $PTM = TM$, then $M$ is flat and totally geodesic.
\item If $PTM \subset T^\perp M$, then either~$M$ is totally geodesic with constant Gaussian curvature~$2/3$
 or~$M$ has constant Gaussian curvature~$5/18$.
 \end{enumerate}
\end{theorem}
\begin{proof}
 Let $v \in T_p M$ be a unit tangent vector. By our assumption, Codazzi's equation
 says that $\tilde R(v,Jv)v$ is a multiple of $Jv$. Once again we choose $v$ such that $g(Pv,v)$
 is maximal on the unit tangent space at~$p$. Then $g(Pv,Jv)=0$ and the Gauss equation becomes
 \[
   R(v,Jv)v = -\frac{2}{3}\bigl( Jv + g(Pv,v)PJv\bigr) + 2 JA_{h(v,v)}v.
 \]
 We now consider two cases.

 \textbf{Case 1: $g(Pv,v)\neq 0$.} By the Gauss equation $PJv$ has to be tangent. As this vector is orthogonal to $v$, 
  we conclude that~$PJv$ is a non-zero multiple of $Jv$ and thus $g(PJv,Jv)= \pm 1$. But then $g(Pv,v)PJv = -Jv$ and
  \[
   K = -R(v,Jv,v,Jv) = -2\|h(v,v)\|^2.
  \]
  Therefore  $\|A_{h(v,v)}v\|^2= \|h(v,v)\|^4= K^2/4$.
  Furthermore, since $Pv$ and $PJv$ are tangent vectors, one obtains
  \[
     g\bigl(\tilde R(v,Jv)h(v,v),Jh(v,v)\bigr) = -\frac{1}{6}\|h(v,v)\|^2.
  \]
  Then Ricci's equation is
  \begin{align*}
   g\bigl(R^\perp(v,Jv)h(v,v),Jh(v,v)\bigr) &= g\bigl(\tilde R(v,Jv)h(v,v),Jh(v,v)\bigr)\\
                                               &\quad  + g(A_{h(v,v)}A_{Jh(v,v)}v,Jv)- g(A_{Jh(v,v)}A_{h(v,v)}v,Jv) \\
                                            &=   -\frac{1}{6}\|h(v,v)\|^2 -2\|A_{h(v,v)}v\|^2 \\
                                            &= \frac{1}{12}K - \frac{1}{2}K^2.
  \end{align*}
  On the other hand the Ricci identity gives
  \begin{align*}
   g\bigl(R^\perp(v,Jv)h(v,v),Jh(v,v)\bigr) &= 2g(h(R(v,Jv)v,v),Jh(v,v)) \\
                                            &= -2Kg(Jh(v,v),Jh(v,v))\\
                                            &= K^2.
  \end{align*}
  Combining the Ricci equation and Ricci identity gives the quadratic equation
  \[
    \frac{3}{2}K^2  -\frac{1}{12}K = 0.
  \]
  Hence $K=0$ since $K = -2\|h(v,v)\|^2$ cannot be positive.

  \textbf{Case 2: $g(Pv,v)= 0$.} We shall proceed in a similar way as in the previous case.
  If $g(Pv,v)=0$, then $P$ clearly maps
  tangent vectors into normal ones. The Gauss equation gives
  \[
     K= \frac{2}{3} - 2\|h(v,v)\|^2.
  \]
  The Ricci equation gives
  \begin{align*}
   g\bigl(R^\perp(v,Jv)h(v,v),Jh(v,v)\bigr) &= -\frac{1}{6}\|h(v,v)\|^2 - 2\| A_{h(v,v)}v\|^2\\
                                             &\quad +\frac{2}{3}\left(g(PJv,h(v,v))^2+g(Pv,h(v,v))^2\right)\\
                                            &= -\frac{1}{6}\|h(v,v)\|^2 - 2\|h(v,v)\|^4\\
                                            &= -\frac{1}{2}K^2 + \frac{3}{4}K - \frac{5}{18}
  \end{align*}
  by Lemma~\ref{lem:PH}, and the Ricci identity becomes
  \[
  g\bigl(R^\perp(v,Jv)h(v,v),Jh(v,v)\bigr) = -2K\|h(v,v)\|^2
                                           = K^2 - \frac{2}{3}K.
  \]
  Thus we have the equation
  \[
    \frac{3}{2}K^2 - \frac{17}{12}K + \frac{5}{18}=0.
  \]
  The roots are $2/3$ and $5/18$. This proves the theorem.
\end{proof}

We note that both cases occuring in Theorem~\ref{thm:curvature} will be improved by later results: case 1 will be improved by 
Theorem~\ref{thm:flat} and case 2 by Theorem~\ref{thm:totallygeo}.

Next we are going to study almost complex surfaces more systematically. In order to do so we will use isothermal coordinates on the surface.
We will use these coordinates amongst other tools to show that
an almost complex submanifold~$M$ such that $PTM\subset T^\perp M$
locally corresponds to an associated constant mean curvature (CMC) surface in
Euclidean 3-space~$\mathbb{R}^3$. Furthermore, the metrics on the almost complex surface
and its associated CMC~surface are equal up to a factor~$2$.
This is the content of Theorem~\ref{thm:epsilon} and~Corollary~\ref{cor:metrics}.

In these computations we will in particular use that for imaginary quaternions we have that
$$xy= -x\cdot y + x \times y,$$
where $\cdot$ is the usual inner product on $\mathbb R^3$ and $\times$ is the vector product on $\mathbb R^3$.

Let $\phi\colon M\to \nks\colon (u,v) \mapsto\bigl(p(u,v),q(u,v)\bigr)$
be an almost complex immersion, where $(u,v)$ are isothermal coordinates on the surface $M$. We write $\phi_u=(p_u, q_u)$ and $\phi_v = (p_v,q_v)$.
Since the coordinates are isothermal, we may assume that~$\phi_v= J\phi_u$ by interchanging $u$ and $v$, if necessary.
Furthermore, as $p$ and $q$ are unit length, there are well defined local functions
$\tilde \alpha$,~$\tilde\beta$,~$\tilde\gamma$ and~$\tilde\delta$  from~$M$ to~$\mathbb{R}^3$ such that
\begin{align*}
 p_u = p\tilde\alpha, \qquad p_v &= p\tilde\beta,  \qquad q_u = q\tilde\gamma, \qquad q_v = q\tilde\delta.
\end{align*}
Then $\phi_v=J\phi_u$ gives
\[
   (p\tilde\beta, q\tilde\delta) = \frac{1}{\sqrt{3}}\bigl(p(2\tilde\gamma-\tilde\alpha), q(-2\tilde\alpha+\tilde\gamma)\bigr),
\]
or 
\begin{equation}
\label{eq:gammadelta}
\tilde\gamma = \frac{\sqrt{3}}{2}\tilde\beta + \frac{1}{2} \tilde\alpha, \qquad  \qquad
\tilde\delta = \frac{1}{2} \tilde\beta -\frac{\sqrt{3}}{2}\tilde\alpha.
\end{equation}
The integrability condition~$p_{uv}=p_{vu}$ yields
\begin{equation}
 \label{eq:de}
    \tilde\alpha_v - \tilde\beta_u = 2\tilde\alpha\times\tilde\beta.
\end{equation}
The other integrability condition~$q_{uv}=q_{vu}$ gives $\tilde\gamma_v - \tilde\delta_u = 2\tilde\gamma\times\tilde\delta$,
which in terms of~$\tilde\alpha$ and~$\tilde\beta$ becomes
\[
  \tilde\alpha_u +\tilde\beta_v = \frac{2}{\sqrt{3}}\tilde\alpha\times\tilde\beta.
\]
Now we write $\alpha=\cos\theta \tilde\alpha+\sin \theta\tilde \beta$ and
$\beta=-\sin \theta \tilde\alpha +\cos\theta\tilde\beta$,
where $\theta=2\pi/3$; i.e.\ we rotate~$\tilde\alpha$ and~$\tilde\beta$ over $2\pi/3$ radians. 
The two previous equations become
\begin{align}
 \alpha_v &= \beta_u, \label{eq:nr1}\\
 \alpha_u +\beta_v &= -\frac{4}{\sqrt{3}}\alpha\times \beta. \label{eq:nr2}
\end{align}

\begin{lemma}
The pull back of the one-form $\alpha\, du + \beta\,dv$ is a well defined closed one form on $M$.
\end{lemma}
\begin{proof}
The differential form $\alpha\,du + \beta\,dv$ is the composite of the
form $p^{-1}dp$ preceded by rotation in the tangent spaces by $2\pi/3$, and as
such its pullback is globally defined and hence the lemma holds.
\end{proof}

Assume now that $M$ is simply connected. In that case, we know that any closed $1$-form is automatically exact. Hence there exists a function~$\epsilon$ such that $\epsilon_u = \alpha$,~$\epsilon_v=\beta$ and
\begin{equation}
 \label{eq:DE}
  \epsilon_{uu}+\epsilon_{vv} = -\frac{4}{\sqrt{3}}\epsilon_u \times \epsilon_v.
\end{equation}
This equation is known as the $H$-surface equation (cf.\cite{wente}).
Of course, as we started with isothermal coordinates we must have that $\epsilon_u^2 +\epsilon_v^2 \ne 0$. 

Note that the converse also holds. Indeed, given a solution of the $H$-surface equation, 
which (see \cite{wente}) can be seen as an equation on a surface, we can define $\alpha=\epsilon_u$ and~$\beta=\epsilon_v$. 
By rotating~$\alpha$ and~$\beta$ we get~$\tilde\alpha = \cos(2\pi/3) \alpha - \sin(2\pi/3) \beta$ 
and~$\tilde\beta=\sin(2\pi/3) \alpha + \cos(2\pi/3)\beta$.
The relations~\eqref{eq:gammadelta} then give~$\tilde \gamma$ and~$\tilde \delta$.
Finally by solving the linear first order system of differential equations~\eqref{eq:de}
we get an almost complex surface in $\nks$. 

Note also that changing the almost complex surface by an isometry~$(p,q) \mapsto (apc^{-1},bqc^{-1})$,
where~$a,b,c$ are unit quaternions, implies that
\begin{align*}
&\alpha^* = c \alpha c^{-1}\\
&\beta^* = c \beta c^{-1},
\end{align*}
where we  denote the new objects by adding a ${}^*$. 
Since $S^3$ is the double cover of~$SO(3)$ (see e.g.\ \cite[p.\, 3]{burstall}) we can represent every element of~$SO(3)$ as conjugation
by a unit quaternion, determined up to changing sign.
Therefore $\alpha$ and $\beta$ change by a rotation, and after integration $\epsilon$ changes by an isometry of $\mathbb{R}^3$.

Conversely, 
applying an Euclidean isometry to the surface~$\epsilon$ gives~$c\epsilon c^{-1} + d$, for 
some unit quaternion~$c$ and an imaginary quaternion~$d$. Deriving this expression with respect to~$u$ and~$v$
we get~$c\alpha c^{-1}$ and~$c\beta c^{-1}$. Performing a rotation over $2\pi/3$ and using~\eqref{eq:gammadelta},
we see that~$\tilde\alpha$, $\tilde \beta$, $\tilde \gamma$ and~$\tilde \delta$ change by conjugation with~$c$.
We obtain the value of~$c$ and then integrating the system of differential equations~\eqref{eq:de} will give
solutions, up to the choice of initial conditions.
This choice of initial conditions determines the unit quaternions~$a$ and~$b$ in the isometry~$(p,q)\mapsto (apc^{-1},bqc^{-1})$
of~$\nks$.
Finally note that changing the sign of $a$, $b$ and $c$ does not change the almost complex surface, 
implying that the almost complex surface does not depend on the choice of the sign of $c$. Therefore, we have shown the following theorem:
\begin{theorem}
 \label{thm:correspondencethm}
There is a one-to-one correspondence between almost complex surfaces in~$\nks$ and solutions of the general $H$-system equation. 
Moreover, two solutions are congruent in $\mathbb R^3$ if and only if the associated solutions in~$\nks$ are congruent.
\end{theorem} 

We now introduce the differential $\Lambda\, dz^2 = g(P\phi_z, \phi_z)\,dz^2$.
Before proving our main results, we show that $\Lambda\,dz^2$ is a globally
defined holomorphic differential.

\begin{lemma}
\label{lem:CR}
The following Cauchy-Riemann equations hold:
\begin{align*}
 (\ab)_u &=\frac{1}{2}(\aa-\bb)_v,\\
 (\ab)_v &=-\frac{1}{2}(\aa-\bb)_u.
\end{align*}
\end{lemma}
\begin{proof}
 Multiplying equations~\eqref{eq:nr1} and~\eqref{eq:nr2} with~$\alpha$ and~$\beta$ gives
 \begin{align*}
  \alpha_v \cdot\alpha -\beta_u\cdot\alpha &=0, &       \beta_v\cdot\alpha +\alpha_u\cdot\alpha &=0, \\
  \alpha_v \cdot \beta -\beta_u \cdot \beta &=0, &      \beta_v\cdot\beta + \alpha_u\cdot\beta &= 0.
 \end{align*}
 The proof immediately follows.
\end{proof}

\begin{lemma}
\label{lem:hol}
 The pull back of $\Lambda\, dz^2$ is a holomorphic differential which is globally defined on $M$. 
\end{lemma}
\begin{proof}
 Using $\phi_v = J\phi_u$, one gets
  \begin{align*}
   4 \Lambda &= g(P\phi_u -iP\phi_v, \phi_u -i\phi_v)\\
             &= 2g(P\phi_u, \phi_u) - 2ig(P\phi_u, J\phi_u),
  \end{align*}
  i.e.,\ $2\Lambda=g(P\phi_u, \phi_u) - ig(P\phi_u,J\phi_u)$. Recall that
   \begin{align*}
     \phi_u &= \Bigl(p\alpha, q\Bigl(\frac{\sqrt{3}}{2}\beta + \frac{1}{2}\alpha\Bigl)\Bigr)\\
     J \phi_u &= \phi_v = \Bigl(p\beta ,q\Bigl(\frac{1}{2}\beta -\frac{\sqrt{3}}{2}\alpha\Bigr)\Bigr).
   \end{align*}
  A simple calculation using the definition of the metric~$g$ and~$P$
  then gives the real and imaginary parts of~$\Lambda$:
  \begin{align}
  \label{eq:lambdaparts}
    \begin{split}
   \re \Lambda &=\frac{1}{4}(\aa -\bb) + \frac{\sqrt{3}}{2}\ab,\\
   \im \Lambda &=\frac{\sqrt{3}}{4}(\aa -\bb) - \frac{1}{2}\ab.
    \end{split}
  \end{align}
  From Lemma~\ref{lem:CR} it follows that $(\re \Lambda)_u = (\im \Lambda)_v$ and
  $(\re \Lambda)_v = -(\im \Lambda)_u$.
  Hence the Cauchy-Riemann equations for $\Lambda=g(P\phi_z,\phi_z)$ hold, so $\Lambda\, dz^2$
  is indeed a holomorphic differential.
  
 Changing isothermal coordinates, we deduce that it is independent of the choice of isothermal coordinates and
 therefore defines a global holomorphic differential on $M$. Note that~$M$ is not required to be simply
 connected.
\end{proof}

\begin{lemma} \label{lem:equiv}
 Let $M$ be an almost complex surface in $\nks$. Then the following are equivalent:
 \begin{enumerate}
 \item $PTM\subset T^\perp M$; 
 \item $\Lambda\, dz^2 = 0$; and
 \item $\aa =\bb$ and $\ab=0$.
\end{enumerate}
\end{lemma}
\begin{proof}
 The almost product structure~$P$ maps tangent vectors into normal vectors if and only if~$g(P\phi_u,\phi_u)$
 and $g(P\phi_u, \phi_v)$ are zero. But $2\Lambda=g(P\phi_u, \phi_u) - ig(P\phi_u,J\phi_u)$, thus the first and second
 assertion are equivalent. Furthermore, $g(P\phi_u,\phi_u)=0$
 and $g(P\phi_u, \phi_v)=0$ if and only if the equations~\eqref{eq:lambdaparts} are zero if and only if~$\aa=\bb$ and~$\ab=0$.
 Thus all assertions are equivalent.
\end{proof}

The following corollary follows immediately from the previous lemma and the fact that a holomorphic differential on a 2-sphere
vanishes. 
\begin{corollary}
\label{cor:lambdazero}
 If we have an almost complex~2-sphere~$S^2$ in $\nks$, then $PTM \subset T^\perp M$.
\end{corollary}

\begin{theorem}
 \label{thm:epsilon}
 The coordinates~$(u,v)$ are isothermal on~$\epsilon$ iff $\Lambda\,dz^2$ vanishes.
 In this case $\epsilon$ corresponds to a  surface in~$\mathbb{R}^3$ with constant mean curvature~$H=-2/\sqrt{3}$.
\end{theorem}
\begin{proof}
 Since $\epsilon_u=\alpha$ and $\epsilon_v=\beta$, the first assertion follows from
Lemma~\ref{lem:equiv}. From equation~\eqref{eq:DE} we know that
 \[
  2H\epsilon_u\times\epsilon_v = \epsilon_{uu}+\epsilon_{vv} =-\frac{4}{\sqrt{3}}\epsilon_u\times\epsilon_v
 \]
 This proves the theorem.
\end{proof}

\begin{corollary}
 \label{cor:metrics}
Let $g$ be the induced metric on an almost complex surface~$M$
in~$\nks$ and $g'$ the metric on the associated surface in~$\mathbb{R}^3$.
If $\Lambda\,dz^2=0$, then $g=2 g'$.
\end{corollary}
\begin{proof}
 If $g$ is the induced metric on $M$, then $g(\phi_u,\phi_u) = \aa+\bb$, which is equal
 to~$2 \aa$ by our assumption. Recall that $\epsilon_u = \alpha$,
 $\epsilon_v=\beta$ and so the corollary follows.
\end{proof}

Now we are able to prove the remaining main results.

\begin{theorem}
\label{thm:totallygeo}
If $M$ is an almost complex surface of $\nks$  with parallel second fundamental form, then $M$ is totally geodesic.
\end{theorem}
\begin{proof}
 Suppose $M$ is not totally geodesic. Then the associated CMC~surface~$\epsilon$
 has Gaussian curvature~$\tfrac{5}{9}$ by Theorem~\ref{thm:curvature}
 and~Corollary~\ref{cor:metrics}.
 But this is not possible since a surface in~$\mathbb{R}^3$ with constant curvature and constant mean curvature
 is either a plane, a circular cylinder or a sphere. The first two examples have curvature $0$, whereas the last one is totally umbilical and therefore, by Theorem~\ref{thm:epsilon}, has curvature $H^2=\tfrac 43$. The corresponding almost complex surface then has constant curvature $\tfrac 23$. 
\end{proof}

\begin{theorem}
 \label{thm:geodesicS2}
 An almost complex topological 2-sphere~$S^2$ in the nearly~K\"ahler~$\nks$ is totally geodesic.
\end{theorem}
\begin{proof}
 By Lemma~\ref{lem:equiv} the differential $\Lambda\, dz^2$ vanishes, so
  we have a CMC~$2$-sphere in~$\mathbb{R}^3$.
  This is a round sphere (by a theorem of~H.~Hopf), hence it is totally umbilical.
  Therefore the Gauss curvature of the CMC~$2$-sphere is~$H^2=4/3$.
 Hence the Gauss curvature of the almost complex sphere in $\nks$ is $2/3$.
 The Gauss equation then says
 \[
   2\|h(v,v)\|^2 = \frac{2}{3} - K = 0,
 \]
 so the topological $2$-sphere is totally geodesic.
\end{proof}

\begin{remark}
From this theorem it follows that a compact almost complex surface~$M$ with Gaussian curvature~$K\geq 0$
has constant curvature~$0$ or~$\tfrac{2}{3}$. Indeed, if the curvature on~$M$ is not identically zero
then by Gauss-Bonnet~$M$ is a two-sphere. Then by the previous theorem~$M$ is totally geodesic
and has curvature~$\tfrac{2}{3}$.
\end{remark}

\section{Examples}

In this last section we discuss two examples of totally geodesic almost complex
surfaces in~$\nks$.

\begin{example}
Consider the immersion
\begin{align*}
 f\colon \mathbb{R}^2 \to \nks \colon (s,t) \mapsto (\cos s + i\sin s, \cos t + i \sin t).
\end{align*}
Then we have
\begin{align*}
 f_s &= (-\sin s + i\cos s, 0),\\
 f_t &= (0, -\sin t + i\cos t),\\
 Jf_s &= \frac{1}{\sqrt{3}}\bigl(\sin s- i \cos s, 2(\sin t - i \cos t)\bigr),\\
 Jf_t &= \frac{1}{\sqrt{3}}\bigl(-2(\sin s- i \cos s), -\sin t + i \cos t\bigr).
\end{align*}
Hence the immersion~$f$ is almost complex. Furthermore, $Pf_s = f_t$, so the almost product
structure maps tangent vector to tangent vectors.
Also, $g(f_s,f_s) = g(f_t,f_t) = \frac{4}{3}$ and  $g(f_s,f_t) = -\frac{2}{3}$
are constant, so~$f$ is flat. A calculation gives $\tilde R(f_s,f_t,f_t,f_s) = 0$, so
that by the Gauss equation and equation~\eqref{eq:ac1} this immersion is totally geodesic as well.
\end{example}

We now show that the above example is the only almost complex surface for which the almost product structure $P$ maps tangent vectors to tangent vectors.
\begin{theorem} 
\label{thm:flat}
Let $M$ be an almost complex surface for which $P$ preserves the tangent space. 
Then $M$ is locally congruent with the immersion
\begin{align*}
 f\colon \mathbb{R}^2 \to \nks \colon (s,t) \mapsto (\cos s + i\sin s, \cos t + i \sin t).
\end{align*}
\end{theorem}
\begin{proof} 
The endomorphism $P$ maps tangent vectors to tangent vectors, is symmetric and compatible with the metric
and anti-commutes with~$J$. From this, it follows that~$P$ at every point of~$M$ has two different
eigenvalues, so we can construct a global orthonormal frame $e_1,e_2$ such that 
\begin{align*}²²		
&Pe_1=e_1\\
&Pe_2=-e_2.
\end{align*}
However it now follows that
\begin{equation*}
0=(\nabla_X P)e_1 =\nabla_X e_1 -P \nabla_X e_1=2 \nabla_X e_1.
\end{equation*}
In the last equation we used that $g(\nabla_X e_1,e_1) = 0$ and~$Pe_2=-e_2$.  
Hence $\nabla_{e_i}e_j=0$, and we know that the immersion is flat and we can choose flat coordinates $u$ and $v$
such that $e_1= \partial_u$ and $e_2=\partial_v$. 
As these coordinates are flat we can use the previous formulas. 

As $P\phi_u=\phi_u$, we must have that
\[\tilde \alpha =\sqrt{3} \tilde\beta.\]
Hence, $\alpha=0$ and $\beta= -2 \tilde \beta$. As $e_1$ and $e_2$ are orthonormal we also have that $\beta$ has constant unit length. 
  
We now fix the initial condition by a rotation in $\mathbb R^3$ 
(or equivalently a conjugation by a unit quaternion~$c$ in $\nks$)  
in such a way that $\epsilon_v(0,0)=\beta(0,0)=(1,0,0)$. Note that $\alpha= \epsilon_u=0$. We then see that the differential equation for the 
$H$-system implies that $\beta$ is constant. 
We also choose initial conditions such that $p(0,0)=(1,0,0,0)$ and $q(0,0)=(1,0,0,0)$. 

It follows that 
\begin{align*}
\tilde \alpha&=\bigl(\frac{\sqrt{3}}{2},0,0\bigr), 
&\tilde \beta&= \bigl(-\frac{1}{2},0,0\bigr),\\
\tilde \gamma&= (0,0,0), 
&\tilde \delta&=(-1,0,0).
\end{align*}

So we get that $q_u= 0$ and $q_v = -qi$, implying that $q= (\cos v,-\sin v,0,0)$. Similarly,
$p_u= p \tfrac{\sqrt{3}i}{2}$ and $p_v=-\tfrac{i}{2}p$ has as solution
\[ 
 p(u,v)=\bigl(\cos(\tfrac{\sqrt{3}}{2}u-\tfrac 12 v),\sin(\tfrac{\sqrt{3}}{2}u-\tfrac 12 v),0,0\bigr).
\]
A change of variable now completes the proof of the theorem. 
\end{proof}

\begin{example}
 Define
\[
   f \colon S^2 \subset \im \mathbb{H} \to \nks \colon x \mapsto \frac{1}{2}(1-\sqrt{3}x, 1+\sqrt{3}x).
\]
In order to do an explicit calculation we choose 
\[
  x(u,v) = (\sin u\cos v, \sin u \sin v, \cos u)
\]
as a parametrization for~$S^2$. Also note that, if we write~$f(u,v)=\bigl(p(u,v),q(u,v)\bigr)$,
we have $p(u,v)(q(u,v))^{-1}=-q(u,v)$ and~$q(u,v)(p(u,v))^{-1}=-p(u,v)$.
A calculation then gives
\begin{align*}
 f_u &= \frac{\sqrt{3}}{2}(-x_u, x_u), 
&  f_v &= \frac{\sqrt{3}}{2}(-x_v, x_v),\\
  Jf_u &= \frac{\sqrt{3}}{2}(-xx_u, xx_u), 
 & Jf_v &= \frac{\sqrt{3}}{2}(-xx_v, xx_v).
\end{align*}
From the parametrization of~$x$ it follows that~$xx_u=\sin u\, x_v$ and~$xx_v= -\sin u\, x_u$. Thus~$Jf_u =\sin u\, f_v$
and~$M=f(S^2)$ is an almost complex surface.
Furthermore using the very definition of~$P$ and the metric~$g$ we obtain~$g(Pf_u,f_u)=0$ 
and~$g(Pf_u,f_v)=0$, thus~$P$ maps tangents vector into normal vectors. Therefore it follows
from the expression of the curvature tensor~$\tilde R$ that the sectional curvature of the
plane spanned by~$f_u$ and~$f_v$ is~$\tfrac{2}{3}$. From the main theorem it follows 
that~$M$ is totally geodesic, so~$M$ has constant curvature~$\tfrac{2}{3}$.
\end{example}

We now put our results together to conclude with the following theorem.
\begin{theorem} Any almost complex surface with parallel second fundamental form is locally congruent 
to one of the above two examples.
\end{theorem}
\begin{proof} 
If $P$ maps tangent vectors to tangent vectors, we obtain the first example by Theorem~\ref{thm:flat}.
So we may assume that $P$ maps tangent vectors into normal vectors and that~$M$ has constant curvature~$\tfrac{2}{3}$
by Theorems~\ref{thm:curvature} and~\ref{thm:totallygeo}.
Furthermore, by Theorem~\ref{thm:epsilon} and~Corollary~\ref{cor:metrics}, 
we can locally associate to~$M$ a surface in~$\mathbb{R}^3$ with constant Gaussian curvature~$\tfrac{4}{3}$ and constant mean curvature 
$H=-\tfrac{2}{\sqrt{3}}$. Hence these surfaces are totally umbilical and therefore mutually congruent. 
The correspondence theorem (Theorem~\ref{thm:correspondencethm}) now completes the proof.
\end{proof}

\nocite{*}
\bibliographystyle{amsplain}
\bibliography{nearlykahlerbib}

\providecommand{\bysame}{\leavevmode\hbox to3em{\hrulefill}\thinspace}
\providecommand{\MR}{\relax\ifhmode\unskip\space\fi MR }
\providecommand{\MRhref}[2]{%
  \href{http://www.ams.org/mathscinet-getitem?mr=#1}{#2}
}
\providecommand{\href}[2]{#2}
\begin{thebibliography}{10}

\bibitem{belgun}
F.~Belgun and A.~Moroianu, \emph{Nearly {K}\"ahler 6-manifolds with reduced
  holonomy}, Ann. Global Anal. Geom. \textbf{19} (2001), no.~4, 307--3019.

\bibitem{boltonpeditwoodward}
J.~Bolton, F.~Pedit, and L.~Woodward, \emph{Minimal surfaces and the affine
  {T}oda field model}, J. Reine Angew. Math. \textbf{459} (1995), 119--150.
  \MR{1319519 (96f:58040)}

\bibitem{bolton}
J.~Bolton, L.~Vrancken, and L.M. Woodward, \emph{On almost complex curves in
  the nearly {K}\"ahler 6-sphere}, Quart. J. Math. Oxford. Ser. (2) \textbf{45}
  (1994), 407--427.

\bibitem{bryant}
R.~L. Bryant, \emph{Submanifolds and special structures on the octonians}, J.
  Differential Geom. \textbf{17} (1982), no.~2, 185--232.

\bibitem{burstall}
F.~E. Burstall~et al., \emph{Conformal {G}eometry of {S}urfaces in ${S}^4$ and
  {Q}uaternions}, Springer, 2002.

\bibitem{butruille}
J.-B. Butruille, \emph{Homogeneous nearly {K}\"ahler manifolds}, in: Handbook
  of {P}seudo-{R}iemannian {G}eometry and {S}upersymmetry, 399--423, RMA Lect.
  Math. Theor. Phys., 16, Eur. Math. Soc., Z\"urich, 2010.

\bibitem{dillenalmostcomplex}
F.~Dillen, L.~Verstraelen, and L.~Vrancken, \emph{Almost complex submanifolds
  of a 6-dimensional sphere {II}}, Kodai Math. J. \textbf{10} (1987), 161--171.

\bibitem{grayalmostcomplex}
A.~Gray, \emph{Almost complex submanifolds of the six sphere}, Proc. Amer.
  Math. Soc. \textbf{20} (1969), 272--276.

\bibitem{graynearlykahler}
\bysame, \emph{Nearly {K}\"ahler manifolds}, J. Differential Geometry
  \textbf{4} (1970), 283--309.

\bibitem{moroianu}
A.~Moroianu, P.~A. Nagy, and U.~Semmelmann, \emph{Unit {K}illing vector fields
  on nearly {K}\"ahler manifolds}, Internat. J. Math. \textbf{16} (2005),
  281--301.

\bibitem{podesta}
F.~Podest\`a and A.~Spiro, \emph{6-dimensional nearly {K}\"ahler manifolds of
  cohomogeneity one}, Journal of Geometry and Physics \textbf{60} (2010),
  no.~6, 156--164.

\bibitem{sekigawa}
K.~Sekigawa, \emph{Almost complex submanifolds of a 6-dimensional sphere},
  Kodai Math. J. \textbf{6} (1983), 147--185.

\bibitem{wente}
H.C. Wente, \emph{Explicit {S}olutions to the {$H$-Surface} {E}quation on
  {T}ori}, Michigan Math. J. \textbf{49} (2001), 501--517.

\bibitem{xufeng}
F.~Xu, \emph{Pseudo-holomorphic curves in nearly {K}\"ahler {${\bf CP}^3$}},
  Differential Geom. Appl. \textbf{28} (2010), no.~1, 107--120. \MR{2579386
  (2011f:53105)}

\end{thebibliography}

\end{document}